\documentclass[english]{amsart}

\usepackage{amsmath,amssymb,enumerate}
\usepackage{amscd,amsxtra,calc}

\usepackage[T1]{fontenc}
\usepackage[all]{xy}

\usepackage{babel}
\usepackage{amstext}
\usepackage{amsmath}
\usepackage{amsfonts}
\usepackage{latexsym}
\usepackage{ifthen}
\usepackage{verbatim}
\usepackage{enumerate}

\usepackage{xypic}
\xyoption{all}
\newtheorem{lemma1}{}[section]

\newenvironment{lemma}{\begin{lemma1}{\bf Lemma.}}{\end{lemma1}}

\newenvironment{theorem}{\begin{lemma1}{\bf Theorem.}}{\end{lemma1}}
\newenvironment{proposition}{\begin{lemma1}{\bf Proposition.}}{\end{lemma1}}

\newenvironment{remark}{\begin{lemma1}{\bf Remark.}\rm}{\end{lemma1}}

\newenvironment{conjecture}{\begin {lemma1}{\bf Conjecture.}}{\end{lemma1}}

\newenvironment{remark*}{{\bf Remark.}}{}
\newenvironment{example*}{{\bf Example.}}{}

\newcommand{\R}{\ensuremath{\mathbb{R}}}
\newcommand{\Q}{\ensuremath{\mathbb{Q}}}

\newcommand{\N}{\ensuremath{\mathbb{N}}}

\newcommand{\merom}[3]{\ensuremath{#1:#2 \dashrightarrow #3}}

\newcommand{\holom}[3]{\ensuremath{#1:#2  \rightarrow #3}}
\newcommand{\fibre}[2]{\ensuremath{#1^{-1} (#2)}}

\makeatletter
\ifnum\@ptsize=0  \fi \ifnum\@ptsize=2  \fi 

\newcommand\sO{{\mathcal O}}

\newcommand{\NAX}{\overline{\mbox{NA}}(X)}

\title{Adjoint $(1,1)$-classes on threefolds}
\date{July 23, 2018}

\author {Andreas H\"oring}

\address{Andreas H\"oring, Universit\'e C\^ote d'Azur, CNRS, LJAD, France}
\email{Andreas.Hoering@unice.fr}

\begin{document}

\begin{abstract} 
We answer a question of Filip and Tosatti concerning a basepoint-free theorem
for transcendental $(1,1)$-classes on threefolds.
\end{abstract}

\maketitle

\section{Introduction}

In a recent preprint Filip and Tosatti proposed the following transcendental version 
of the basepoint-free theorem:

\begin{conjecture} \label{conjectureFT} \cite[Conj.1.2]{FT17}
Let $X$ be a compact K\"ahler manifold, and let $\alpha$ be a nef $(1,1)$-class 
on $X$ such that $\alpha-K_X$ is nef and big. Then the class $\alpha$ is semiample: there
exists a morphism with connected fibres $\holom{\psi}{X}{Z}$ onto a normal compact K\"ahler space
and a K\"ahler class $\alpha_Z$ on $Z$ such that $\alpha = \psi^* \alpha_Z$. 
\end{conjecture}

If $\alpha$ is the class of a $\Q$-divisor (and hence $X$ is projective), this statement is equivalent to the 
basepoint-free theorem (cf. \cite[Thm.3.9.1]{BCHM10} for $\R$-divisors). 
Filip and Tosatti proved this conjecture for surfaces \cite[Thm.1.3]{FT17}. 
The aim of this note is to clarify the situation in dimension three:

\begin{theorem} \label{theoremmain}
Conjecture \ref{conjectureFT} holds if $\dim X=3$ and $\alpha-K_X$ is a K\"ahler class.
\end{theorem}

The assumption is slightly stronger than in the conjecture, but 
should be enough for potential applications to the K\"ahler-Ricci flow.
Theorem \ref{theoremmain} was shown by Tosatti and Zhang \cite[Thm.1.4]{TZ18} when the class $\alpha$
is nef, but not big, so we only have to deal with the case where the morphism $\psi$
is bimeromorphic. More precisely we prove the following:

\begin{theorem} \label{theoremcontractionfaces}
Let $X$ be a normal $\Q$-factorial compact K\"ahler threefold with terminal singularities. 
Let $\omega$ be a K\"ahler class on $X$ such that $\alpha:= K_X+\omega$ is nef and big. 
Then there exists a bimeromorphic morphism $\holom{\psi}{X}{Z}$ onto a normal compact K\"ahler space
with isolated rational singularities and a K\"ahler class $\alpha_Z$
on $Z$ such that $\alpha = \psi^* \alpha_Z$.
\end{theorem}

If the curves contracted by $\psi$ define an extremal {\em ray} in the (generalised) Mori cone of $X$, this is 
the contraction theorem for K\"ahler threefolds (see Remark \ref{remarkMMP}). In the setting of Theorem \ref{theoremcontractionfaces}
we contract, more generally, an extremal {\em face} of the Mori cone. The obvious idea is to reduce to the case
of extremal rays by running a MMP $X \dashrightarrow Y$ where we contract only $\alpha_\bullet$-trivial $K_\bullet$-negative extremal rays. However $Y$ is typically {\em not} the space we are looking for. 
Indeed if one of the steps of the MMP is a flip,
the space $Y$ can contain $\alpha_Y$-trivial curves that are $K_Y$-positive. These curves can not contracted by the MMP, 
so we need some additional argument.

At the moment we do not know how to prove Conjecture \ref{conjectureFT} if the class $\alpha-K_X$ is merely nef and big.
The main reason is that for crepant birational contractions the geometry of the exceptional locus is more complicated,
so the strategy used for $K_X$-negative extremal rays in 
\cite{HP16, CHP16} will not work.
However it is easy to settle a special case:

\begin{proposition} \label{propositionCY}
Let $X$ be a normal $\Q$-factorial compact K\"ahler threefold with terminal singularities
such that $K_X \equiv 0$. 
Let $\alpha$ be a nef and big $(1,1)$-class on $X$.
Then there exists a bimeromorphic morphism $\holom{\psi}{X}{Z}$ onto a normal compact K\"ahler space
with rational singularities and a K\"ahler class $\alpha_Z$
on $Z$ such that $\alpha = \psi^* \alpha_Z$.
\end{proposition}

The proof is based on the observation that the decomposition theorem \cite{CHP16, Dru17} 
essentially reduces the problem to the case of surfaces proven by Filip and Tosatti.
In a similar spirit the Beauville-Bogomolov decomposition \cite{Bea83} reduces Conjecture \ref{conjectureFT} for manifolds
with trivial canonical class to the case of Hyperk\"ahler manifolds. This case is of independent interest (cf. \cite{AV18}),
we hope to come back to it in the future.

{\bf Acknowledgement.} I thank Valentino Tosatti for encouraging me to write up this note.

\section{Definitions and terminology}

We will use standard terminology  
of the minimal model program (MMP) as explained in \cite{KM98} or \cite{Deb01},
see also \cite[Sect.2-3]{HP16} for minor modifications in the K\"ahler setting.
For notions from analytic geometry we refer to \cite{DP04}, \cite{Dem12}.

Let $X$ be a normal compact complex space with at most rational singularities.
Suppose that $X$ is in the Fujiki class, i.e. $X$ is bimeromorphic to a compact K\"ahler manifold. 
A $(1,1)$-class on $X$ is an element of $H^{1,1}_{\rm BC}(X)$,
the Bott-Chern group of $(1,1)$-currents that are locally $\partial \bar \partial$-exact (cf. \cite[Sect.2]{HP16} for details). 
Even if $X$ is projective, the inclusion 
$$
\mbox{NS}(X) \otimes \R \subset H^{1,1}_{\rm BC}(X)
$$
is typically not an equality. However if $H^2(X, \sO_X)=0$, we can apply Kodaira's criterion to a desingularisation
to see that $X$ is projective and every $(1,1)$-class is an $\R$-divisor class.

Let $\alpha \in H^{1,1}_{\rm BC}(X)$ be a $(1,1)$-class. An irreducible curve $C \subset X$ is $\alpha$-trivial
if $\alpha \cdot C=0$.
The class $\alpha$ is big if it contains a K\"ahler current, it is modified K\"ahler
if there exists a modification \holom{\mu}{X'}{X} and a K\"ahler form $\omega$ on $X'$
such that $[\mu_* \omega] = \alpha$. A modified K\"ahler class is big, moreover for
every hypersurface $S \subset X$, the restriction $\alpha|_S$ is big: since
the hypersurface $S$ is not in the image of the $\mu$-exceptional locus,
the restriction of the current $\mu_* \omega$ to $S$ is well-defined and a K\"ahler current. 

Following \cite[Defn.3]{Pau98} one can define when a $(1,1)$-class is nef. A nef $(1,1)$-class is always nef in the algebraic sense, i.e. for every subvariety $Z \subset X$ we have 
$$
Z \cdot \alpha^{\dim Z} \geq 0.
$$
For a nef and big class $\alpha$ on $X$ one defines the Null locus
$$
\mbox{\rm Null}(\alpha) = \cup_{Z \cdot \alpha^{\dim Z} = 0} Z,
$$ 
where the union runs over all positive-dimensional subvarieties.
A priori the Null locus is a countable union of subvarieties, but by a theorem
of Collins and Tosatti \cite{CT15} (applied to the pull-back of the class $\alpha$ to a desingularisation) we know that the Null locus has only finitely many irreducible components.
Moreover we have the following :

\begin{lemma} \label{lemmakaehlerclass}
Let  $X$ be a normal compact complex space with isolated rational singularities.
Let $\alpha \in H^{1,1}_{\rm BC}(X)$ be a nef and big class such that  $Z \cdot \alpha^{\dim Z} > 0$
for every positive-dimensional subvariety. Then $\alpha$ is a K\"ahler class, i.e.
there exists a K\"ahler form $\omega$ on $X$ such that $[\omega]=\alpha$.
\end{lemma}

\begin{proof}
Let \holom{\mu}{X'}{X} be a resolution of singularities. Then $\mu^* \alpha$ is a nef and big
class such that the Null locus coincides with the $\mu$-exceptional locus.
By \cite[Thm.1.1]{CT15} this implies that the non-K\"ahler locus of $\mu^* \alpha$
coincides with the $\mu$-exceptional locus.
Combining Demailly's regularisation \cite{Dem92} and \cite[Thm.3.17(ii)]{Bou04}
this implies that there exists a K\"ahler current $T$ in the class $\mu^* \alpha$
which is singular exactly along the exceptional locus. The push-forward
$\mu_* T$ is thus a K\"ahler current in the class $\alpha$ that is singular at most in the 
finitely many singular points of $X$.  Using a regularized maximum, we obtain the smooth form $\omega$ (cf. the proof of \cite[Thm.1.3]{FT17} for details).
\end{proof}

\section{Bimeromorphic geometry}

The following lemma is a consequence of Araujo's description of the mobile cone \cite[Thm.1.3]{Ara10}, the proof is part of the proof of \cite[Prop.7.11]{HP16}.

\begin{lemma} \label{lemmaaraujo}
Let $F$ be a projective surface such that $H^2(F, \sO_F)=0$. Let $\alpha_F$ be a nef (1,1)-class such that 
$$
\alpha_F^2=0, \quad K_F \cdot \alpha_F<0.
$$
Then $F$ is covered by $\alpha_F$-trivial curves.
\end{lemma}

We will need a slightly stronger version of \cite[Prop.7.11]{HP16}:

\begin{lemma} \label{lemmacovered}
Let $Y$ be a normal $\Q$-factorial compact K\"ahler threefold with terminal singularities. 
Let $\omega_Y$ be a modified K\"ahler class on $Y$ such that $\alpha_Y:= K_Y+\omega_Y$ is nef and big. 
Let $S \subset Y$ be an integral surface such that $\alpha_Y^2 \cdot S=0$. Then $S$ is Moishezon and covered by $\alpha_Y$-trivial curves.
\end{lemma}

\begin{proof}
Let $\holom{\pi}{S'}{S}$ be the composition of normalisation and minimal resolution of $S$, then we have
\begin{equation} \label{sub}
K_{S'} = \pi^* K_S - E
\end{equation}
with $E$ an effective $\Q$-divisor on $S'$ (see \cite[Sect.4.1]{HP16}). Since $S$ is a K\"ahler space, the surface $S'$ is K\"ahler.

{\em 1st case. The pull-back $\pi^* \alpha_Y|_S$ is numerically trivial.}
Then we have $- \pi^* K_Y|_S = \pi^* \omega_Y|_S$. Since $\omega_Y$ is modified K\"ahler the restriction $\omega_Y|_S$
is a big (1,1)-class. Thus $- \pi^* K_Y|_S$ is a big line bundle on $S'$ and $S'$ is Moishezon. In particular $S'$ (and hence $S$) 
is covered by curves, they are all $\alpha_Y|_S$-trivial.

{\em 2nd case. The pull-back $\pi^* \alpha_Y|_S$ is not numerically trivial.}
The nef and big class $\alpha_Y$ defines an intersection form on $N^1(Y)$ which, by the Hodge index theorem, 
has signature $(1,d)$. Since $\alpha_Y^2 \cdot S=0$ this implies that
either $\alpha_Y \cdot S = 0$ or $\alpha_Y \cdot S^2<0$. The former case is excluded since $\pi^* \alpha_Y|_S$ is not numerically trivial.
By the adjunction formula and using $K_Y=\alpha_Y-\omega_Y$ we obtain
$$
K_S \cdot \alpha_Y|_S = (K_Y+S) \cdot S \cdot \alpha_Y = \alpha_Y^2 \cdot S - \alpha_Y \cdot \omega_Y \cdot S + \alpha_Y \cdot S^2.
$$ 
The first term is zero, the other terms are negative, so $K_S \cdot \alpha_Y|_S<0$.
By \eqref{sub} this implies
$$
K_{S'} \cdot \pi^* \alpha_Y|_S \leq K_S \cdot \alpha_Y|_S < 0.
$$
Since $\pi^* \alpha_Y|_S$ is a nef class, this shows that $K_{S'}$ is not pseudoeffective.
In particular we have $H^2(S', \sO_{S'}) = H^0(S', K_{S'})=0$. 
By Kodaira's criterion $S'$ is projective and we conclude with Lemma \ref{lemmaaraujo}.
\end{proof}

\begin{proposition} \label{propositionsmall}
Let $Y$ be a normal compact K\"ahler threefold with isolated rational singularities.
Let $\alpha_Y$ be a (1,1)-class on $Y$ that is nef and big. Suppose that every irreducible component of $\mbox{\rm Null}(\alpha_Y)$ has dimension one. Then there
exists a bimeromorphic morphism $\holom{\mu_Y}{Y}{Z}$ onto a normal compact complex space $Z$ such that 
every connected component of $\mbox{\rm Null}(\alpha_Y)$ is contracted onto a point,
and 
$$
Y \setminus \mbox{\rm Null}(\alpha_Y) \simeq Z \setminus
\mu_Y(\mbox{\rm Null}(\alpha_Y))
$$
is an isomorphism.
\end{proposition}

\begin{remark}
We do not claim that $\alpha_Y$ is a pull-back of a $(1,1)$-class from $Z$. 
In fact this already fails for nef and big divisors on projective surfaces. 
\end{remark}

\begin{proof} The proof of \cite[Thm.7.12]{HP16} applies without changes. Indeed the first part of the proof consists in verifying
that the nef supporting class is nef and big and has one-dimensional non-K\"ahler locus (which is exactly our assumption). The rest
of the proof uses only this property.
\end{proof}

\begin{remark} \label{remarkMMP}
We need a technical remark concerning the use of MMP in our setting:
let $X$ be normal $\Q$-factorial compact K\"ahler threefold with terminal singularities. 
Suppose either that $K_X$ is pseudoeffective (or equivalently, that $X$ is not uniruled),
or that the MRC fibration is an almost holomorphic map $X \dashrightarrow B$ onto a compact K\"ahler surface. Let $\alpha$ be a nef and big class on $X$ and suppose that there exists
a curve $C \subset X$ such that
$$
\alpha \cdot C=0, \ K_X \cdot C<0.
$$
Then there exists a contraction \holom{\mu}{X}{X'} of a $K_X$-negative extremal ray $\Gamma$ such that $\alpha \cdot \Gamma=0$.

If $K_X$ is pseudoeffective, this is completely standard: by the cone theorem \cite[Thm.1.2]{HP16} we can decompose the class
$$
C = \eta + \sum \lambda_i \Gamma_i
$$
where $K_X \cdot \eta \geq 0$, the coefficients $\lambda_i \geq 0$ and $\Gamma_i$ are rational curves generating
extremal rays in the cone $\NAX$. Since $K_X \cdot C<0$ there exists at least one coefficient $\lambda_{i_0}>0$.
Since $\alpha$ is nef and $\alpha \cdot C=0$ we have $\alpha \cdot \Gamma_{i_0} =0$. By the contraction theorem
\cite[Thm.1.3]{HP16} we can contract the extremal ray $\R^+ \Gamma_{i_0}$.

If the base of the MRC fibration has dimension two, we only have weaker forms of the cone and contraction theorem
\cite{HoPe2}. Let $F$ be a general fibre of the MRC-fibration. Then $\alpha \cdot F>0$, since $\alpha$ is nef and big. 
Let $\omega_X$ be a K\"ahler class on $X$.
Then for all $\varepsilon>0$ the class
$$
\omega_\varepsilon := 2 \frac{\alpha + \varepsilon (\alpha \cdot F) \omega_X}{(\alpha + \varepsilon (\alpha \cdot F) \omega_X) \cdot F}
$$
is a normalised K\"ahler class in the sense of \cite[Defn.1.2.]{HoPe2}. Moreover, since $\alpha \cdot C=0$, we know that
for $0 < \varepsilon \ll 1$ the intersection number
$$
(K_X + \omega_\varepsilon) \cdot C < 0.
$$
Thus we can use \cite[Thm.3.13, Thm.3.15]{HoPe2} to conclude as above.
\end{remark}

\begin{proof}[Proof of Theorem \ref{theoremcontractionfaces}]
If $H^2(X, \sO_X)=0$ the nef and big class $\alpha$ is an $\R$-divisor class, so we can conclude with the basepoint-free theorem for $\R$-divisors. Suppose from now on that
$H^2(X, \sO_X) \neq 0$. Then $X$ is either not uniruled or the MRC fibration is an 
almost holomorphic map $X \dashrightarrow B$ onto a compact K\"ahler surface (cf. \cite[Sect.1]{HoPe2}).
Note that these properties are invariant under the minimal model program.

{\em Step 1. Running a MMP.}
We start by running an $\alpha_\bullet$-trivial $K_\bullet$-MMP. More precisely set 
$$
X_0:= X, \qquad \omega_0 := \omega, \qquad \alpha_0 :=\alpha.
$$
We define inductively a sequence of bimeromorphic maps as follows: for $i \in \N$,
suppose that there exists a curve $C \subset X$ such that $K_{X_{i-1}} \cdot C<0$ and $\alpha \cdot C=0$.
By Remark \ref{remarkMMP} 
there exists a contraction $\holom{\varphi_i}{X_{i-1}}{X_i}$ of a $K_{X_{i-1}}$-negative extremal ray $\Gamma_{i-1}$ such that 
$\alpha_{i-1} \cdot \Gamma_{i-1}=0$.
The contraction is not of fibre type, since otherwise $X_{i-1}$ is covered by $\alpha_{i-1}$-trivial curves, contradicting
the hypothesis that $\alpha_{i-1}$ is nef and big.
Thus the contraction is bimeromorphic, and we denote by $\merom{\mu_i}{X_{i-1}}{X_i}$ the divisorial contraction or, for a small ray,
its flip. Since the contraction is $\alpha_{i-1}$-trivial, the push-forward $(\mu_i)_* \alpha_{i-1} =: \alpha_i$ 
is a nef and big $(1,1)$-class $\alpha_i$ on $X_i$. We also set $(\mu_i)_* \omega_{i-1} =: \omega_i$, so
$K_{X_i}  + \omega_i = \alpha_i$. The class $\omega_i$ is not necessarily K\"ahler, but since $(\mu_i)^{-1}$ does not contract any divisors, it is a modified K\"ahler class.  

By Mori's termination of flips for terminal threefolds \cite{Mor88}, 
the MMP terminates after finitely many steps, so composing the $\mu_i$ we obtain 
a bimeromorphic morphism
$$
\mu : X \dashrightarrow Y
$$ 
such that $Y$ is a normal $\Q$-factorial compact K\"ahler threefold with terminal singularities, 
the class $\mu_* \alpha =: \alpha_Y$ is nef and big and
$\alpha_Y = K_Y + \omega_Y$ with $\omega_Y$ a modified K\"ahler class.
Since $Y$ is the outcome of the $\alpha_\bullet$-trivial $K_\bullet$-MMP,
we have $\alpha_Y \cdot C>0$ for every curve $C$ such that $K_Y \cdot C<0$. 

We claim that  $\mbox{\rm Null}(\alpha_Y)$ does not contain an irreducible surface $S$: otherwise 
we know by Lemma \ref{lemmacovered} that $S$ is covered 
$\alpha_Y$-trivial negative curves $(C_t)_{t \in T}$.
Since $\omega_Y$ is modified K\"ahler, the restriction $\omega_Y|_S$ is a big $(1,1)$-class. Thus
for a general curve $C_t$ we have $\omega_Y \cdot C_t>0$. Thus we obtain that $K_Y \cdot C_t<0$,
in contradiction to the preceding paragraph.

{\em Step 2. Construction of $Z$.} By the claim the Null locus of $\alpha_Y$ has pure dimension one.
By Proposition \ref{propositionsmall} there exists thus a bimeromorphic map $\holom{\mu_Y}{Y}{Z}$
contracting the connected components of $\mbox{\rm Null}(\alpha_Y)$ onto points.
We claim that the composition $\merom{\psi := \mu_Y \circ \mu}{X}{Z}$ is holomorphic: 
denote by $\Gamma_\mu \subset X \times Y$ the graph of $\mu$ and by  $\holom{p}{\Gamma_\mu}{X}$
the projection onto $X$.
Denote by $\Gamma_\psi \subset X \times Z$ the graph of $\psi$.
Then $\psi$ is a morphism if and only if the projection $\Gamma_\psi \rightarrow X$
is an isomorphism. 

Since $\psi$ is the composition of the meromorphic map $\mu$
with the holomorphic map $\mu_Y$ this is equivalent to showing that the $p$-fibres
are contained in $\mbox{\rm Null}(\alpha_Y)$\footnote{Since $\Gamma_\mu \subset X \times Y$, the $p$-fibres are naturally embedded in $Y$.}.
Yet for a threefold-MMP, the $p$-fibres are the (strict transforms) of flipped curves.
Since we ran an $\alpha_\bullet$-trivial MMP, the class $\alpha_\bullet$ is trivial on both the contracted and the flipped curves.
Thus the $p$-fibres are $\alpha_Y$-trivial.

Thus we have constructed a bimeromorphic morphism $\holom{\psi}{X}{Z}$ 
such that the restriction of $\alpha$ to every fibre is numerically trivial and the exceptional locus coincides with 
$\mbox{\rm Null}(\alpha)$.
Since $\alpha = K_X+\omega$ and $\omega$ is K\"ahler, this morphism is projective with relatively ample line bundle $-K_X$.
In particular by relative Kodaira vanishing \cite{Anc87} we have $R^j \psi_* \sO_X = 0$ for all $j \geq 0$. 
Since $X$ has terminal, hence rational singularities, this implies that $Z$ rational singularities.
By \cite[Lemma 3.3.]{HP16} there exists thus a (1,1)-class $\alpha_Z$ on $Z$ such that $\alpha = \psi^* \alpha_Z$.
By Lemma \ref{lemmakaehlerclass} the class $\alpha_Z$ is K\"ahler. 
\end{proof}

\begin{remark*}
In general the space $Z$ is not $\Q$-factorial, since $\mu_Y$ is a small contraction. 
Thus $Z$ does not necessarily have terminal singularities.
\end{remark*}

\section{Calabi-Yau case}

\begin{proof}[Proof of Proposition \ref{propositionCY}]
By the decomposition theorem \cite[Thm.1.2]{Dru17} \cite[Thm.9.2]{CHP16} there exists a finite cover
$\holom{\mu}{\tilde X}{X}$ that is \'etale over the nonsingular locus such that $\tilde X$ is either a torus, 
a Calabi-Yau threefold (in particular
$H^2(\tilde X, \sO_{\tilde X}) = 0$) or a product of an elliptic curve $E$ with a K3 surface $S$.

If $\tilde X$ is a torus, the pull-back $\mu^* \alpha$ is K\"ahler. Thus $\alpha$ itself is K\"ahler and we are done.
If $\tilde X$ is a Calabi-Yau, the $(1,1)$-class $\alpha$ is an $\R$-divisor class and we conclude by the basepoint-free theorem. 

Thus we are left to deal with the case $\tilde X \simeq E \times S$. Set $\tilde \alpha := \mu^* \alpha$. For a K3 surface we have $H^0(S, \Omega_S)=0$, which immediately implies 
$$
H^{1,1}_{\rm BC}(X) = p_E^* H^{1,1}_{\rm BC}(E) \times p_S^* H^{1,1}_{\rm BC}(S),
$$
where $\holom{p_E}{\tilde X}{E}$ and $\holom{p_S}{\tilde X}{S}$ are the projections on the factors. Thus we can write
$$
\tilde \alpha = \lambda F + p_S^* \alpha_S
$$
where $F$ is a $p_E$-fibre, $\lambda \in \R$, and $\alpha_S$ is a $(1,1)$-class on $S$. The restriction of the nef and big class
$\tilde \alpha$ to a general $p_E$-fibre is nef and big, so $\alpha_S$ is nef and big. 
The restriction of the nef and big class
$\tilde \alpha$ to a general $p_S$-fibre is nef and big, so $\lambda>0$.

Let now $Z$ be an irreducible component on the null locus of $\tilde \alpha$. If $Z$ is a curve, we have
$$
0 = \tilde \alpha \cdot Z = \lambda F \cdot Z + \alpha_S \cdot (p_S)_* Z.
$$ 
By what precedes we obtain $F \cdot Z=0$, so $Z$ is contained in a fibre $\fibre{p_E}{t_0}$. 
Yet then $Z$ trivially deforms in a family $Z_t = (t \times Z)_{t \in E}$, and $\alpha \cdot Z_t=0$ for all the curves in this 
family. Thus $Z$ is not an irreducible component of the null locus.

This shows that any irreducible component of the null locus is a surface $Z$. If the map $p_S|_Z : Z \rightarrow S$
is generically finite, the pull-back $p_S|_Z^* \alpha_S$ is big. Hence $\tilde \alpha|_Z$ is nef and big, but this contradicts
the property of being in the null locus. Thus $p_S(Z)$ is an irreducible curve, and since 
$$
Z \subset \fibre{p_S}{p_S(Z)} = E \times p_S(Z)
$$
is irreducible, we see that $Z= E \times p_S(Z)$. 
Since $Z$ is in the null locus we have
$$
0 = \tilde \alpha^2 \cdot Z = \lambda^2 F^2 \cdot Z + 2 \lambda F \cdot  p_S^* \alpha_S \cdot Z + (p_S^* \alpha_S)^2 \cdot Z
= 2 \lambda \alpha_S \cdot p_S(Z).
$$
Since $\lambda>0$ we see that $p_S(Z)$ is in the null-locus of $\alpha_S$, hence it is a $(-2)$-curve (cf. the proof of \cite[Thm.1.3]{FT17}). In conclusion we obtain that
$$
\mbox{\rm Null}(\tilde \alpha) = E \times \mbox{\rm Null}(\alpha_S)
$$
By \cite[Thm.1.3]{FT17} there exists a bimeromorphic map $\holom{\psi_S}{S}{S'}$ onto a normal compact surface $S'$
and a K\"ahler class $\alpha_{S'}$ on $S'$ such that $\alpha_S = \psi_S^* \alpha_{S'}$.
We then set
$$
\tilde \psi := id_E \times \psi_S : \tilde X = E \times S \rightarrow E \times S' =: \tilde X'
$$
and $\lambda F + \alpha_{S'}$ is a K\"ahler class on $\tilde X'$ such that $\tilde \psi^* (\lambda F + \alpha_{S'}) = \tilde \alpha$.

Let us now show that this map descends to $X$: up to replacing $\mu$ by its Galois closure, we have $X = \tilde X/G$ where $G$ is the Galois group of $\mu$.  Any automorphism  $f$ on $E \times S$ is of the form $f_E \times f_S$ \cite[p.8]{BeaKatata}, thus
$G$ acts on the factors $E$ and $S$.  The class $\mu^* \alpha$ is $G$-invariant,
so $\alpha_S$ is invariant under the $G$-action on $S$. The Null locus $\mbox{Null}(\alpha_S)$ being $G$-invariant,
we see that there is an induced $G$-action on $S'$ that makes $\psi_S$ is $G$-equivariant. 
Since $\tilde \psi = id_E \times \psi_S$ there is an induced $G$-action on $\tilde X'$ that makes $\tilde \psi$ is $G$-equivariant. 
We set
$Z := \tilde X'/G$ and denote by $\holom{\psi}{X=X'/G}{Z=\tilde X'/G}$ the bimeromorphic morphism
induced by $\tilde \psi$.
Since $\psi$ is crepant and $X$ has terminal singularities, the space $Z$ has canonical, hence rational singularities.

Denote by $\holom{\mu'}{\tilde X'}{Z=\tilde X'/G}$ the finite cover. 
Since $\mu^* \alpha$ is $G$-invariant, 
the K\"ahler class $\lambda F + \alpha_{S'}$ is $G$-invariant and defines a K\"ahler class $\alpha_Z$ on $Z$.
By construction we have $\alpha = \psi^* \alpha_Z$. 
\end{proof}

\begin{remark*}
If $\tilde X \simeq E \times S$, then $X$ is in general not a product. Let $f_E$ be a fixpoint free involution (e.g. a translation by
$2$-torsion point), and $f_S$ an involution with fixed points. Set $X := \tilde X/\langle f_E \times f_S \rangle$ and denote by
$p : X \rightarrow S/\langle f_S\rangle$ the map induced by the projection $p_S$. Then $p$ has multiple fibres over the fixed points, so $X$ is not a product.
\end{remark*}


\begin{thebibliography}{BCHM10}

\bibitem[Anc87]{Anc87}
Vincenzo Ancona.
\newblock Vanishing and nonvanishing theorems for numerically effective line
  bundles on complex spaces.
\newblock {\em Ann. Mat. Pura Appl. (4)}, 149:153--164, 1987.

\bibitem[Ara10]{Ara10}
Carolina Araujo.
\newblock The cone of pseudo-effective divisors of log varieties after
  {B}atyrev.
\newblock {\em Math. Z.}, 264(1):179--193, 2010.

\bibitem[AV18]{AV18}
Ekaterina Amerik and Misha Verbitsky.
\newblock M{B}{M} loci in families of {H}yperk{\"a}hler manifolds and centers
  of birational contractions.
\newblock {\em arXiv preprint}, 1804.00463, 2018.

\bibitem[BCHM10]{BCHM10}
Caucher Birkar, Paolo Cascini, Christopher~D. Hacon, and James McKernan.
\newblock Existence of minimal models for varieties of log general type.
\newblock {\em J. Amer. Math. Soc.}, 23(2):405--468, 2010.

\bibitem[Bea83a]{BeaKatata}
Arnaud Beauville.
\newblock Some remarks on {K}\"ahler manifolds with {$c_{1}=0$}.
\newblock In {\em Classification of algebraic and analytic manifolds ({K}atata,
  1982)}, volume~39 of {\em Progr. Math.}, pages 1--26. Birkh\"auser Boston,
  Boston, MA, 1983.

\bibitem[Bea83b]{Bea83}
Arnaud Beauville.
\newblock Vari\'et\'es {K}\"ahleriennes dont la premi\`ere classe de {C}hern
  est nulle.
\newblock {\em J. Differential Geom.}, 18(4):755--782 (1984), 1983.

\bibitem[Bou04]{Bou04}
S{\'e}bastien Boucksom.
\newblock Divisorial {Z}ariski decompositions on compact complex manifolds.
\newblock {\em Ann. Sci. \'Ecole Norm. Sup. (4)}, 37(1):45--76, 2004.

\bibitem[CHP16]{CHP16}
Fr\'ed\'eric Campana, Andreas H\"oring, and Thomas Peternell.
\newblock Abundance for {K}\"ahler threefolds.
\newblock {\em Ann. Sci. \'Ec. Norm. Sup\'er. (4)}, 49(4):971--1025, 2016.

\bibitem[CT15]{CT15}
Tristan~C. Collins and Valentino Tosatti.
\newblock {K}{\"a}hler currents and null loci.
\newblock {\em Invent.Math.}, 202(3):1167--1198, 2015.

\bibitem[Deb01]{Deb01}
Olivier Debarre.
\newblock {\em Higher-dimensional algebraic geometry}.
\newblock Universitext. Springer-Verlag, New York, 2001.

\bibitem[Dem92]{Dem92}
Jean-Pierre Demailly.
\newblock Regularization of closed positive currents and intersection theory.
\newblock {\em J. Algebraic Geom.}, 1(3):361--409, 1992.

\bibitem[Dem12]{Dem12}
Jean-Pierre Demailly.
\newblock {\em Analytic methods in algebraic geometry}, volume~1 of {\em
  Surveys of Modern Mathematics}.
\newblock International Press, Somerville, MA; Higher Education Press, Beijing,
  2012.

\bibitem[DP04]{DP04}
Jean-Pierre Demailly and Mihai Paun.
\newblock Numerical characterization of the {K}\"ahler cone of a compact
  {K}\"ahler manifold.
\newblock {\em Ann. of Math. (2)}, 159(3):1247--1274, 2004.

\bibitem[Dru18]{Dru17}
St\'ephane Druel.
\newblock A decomposition theorem for singular spaces with trivial canonical
  class of dimension at most five.
\newblock {\em Invent. Math.}, 211(1):245--296, 2018.

\bibitem[FT17]{FT17}
Simion Filip and Valentino Tosatti.
\newblock Smooth and rough positive currents.
\newblock {\em arXiv preprint}, 1709.05385, 2017.

\bibitem[HP15]{HoPe2}
Andreas H{\"o}ring and Thomas Peternell.
\newblock Mori fibre spaces for {K}\"ahler threefolds.
\newblock {\em J. Math. Sci. Univ. Tokyo}, 22(1):219--246, 2015.

\bibitem[HP16]{HP16}
Andreas H{\"o}ring and Thomas Peternell.
\newblock Minimal models for {K}\"ahler threefolds.
\newblock {\em Invent. Math.}, 203(1):217--264, 2016.

\bibitem[KM98]{KM98}
J{\'a}nos Koll{\'a}r and Shigefumi Mori.
\newblock {\em Birational geometry of algebraic varieties}, volume 134 of {\em
  Cambridge Tracts in Mathematics}.
\newblock Cambridge University Press, Cambridge, 1998.
\newblock With the collaboration of C. H. Clemens and A. Corti.

\bibitem[Mor88]{Mor88}
Shigefumi Mori.
\newblock Flip theorem and the existence of minimal models for {$3$}-folds.
\newblock {\em J. Amer. Math. Soc.}, 1(1):117--253, 1988.

\bibitem[Pau98]{Pau98}
Mihai Paun.
\newblock Sur l'effectivit\'e num\'erique des images inverses de fibr\'es en
  droites.
\newblock {\em Math. Ann.}, 310(3):411--421, 1998.

\bibitem[TZ18]{TZ18}
Valentino Tosatti and Yuguang Zhang.
\newblock Finite time collapsing of the {K}{\"a}hler-{R}icci flow on
  threefolds.
\newblock {\em Ann. Sc. Norm. Super. Pisa Cl. Sci.}, pages 105--118, 2018.

\end{thebibliography}

\end{document}